\documentclass[letterpaper, 11pt]{article}

\usepackage{etoolbox}
\newtoggle{SODA} 
\toggletrue{SODA}
\togglefalse{SODA}

\usepackage[margin=1in]{geometry}
\usepackage{amsthm}

\usepackage{dsfont}     
\usepackage{amssymb}     
\usepackage{bbm}        
\usepackage{bm}
\usepackage[scr]{rsfso} 	
\DeclareMathAlphabet{\mathbmscr}{OMS}{mdugm}{b}{n} 

\usepackage{mathtools}

\usepackage{algorithm}
\usepackage{algcompatible}
\usepackage{algpseudocode}

\usepackage{amsmath}

\usepackage{nomencl}
\makenomenclature

\usepackage{enumerate}

\usepackage{paralist}

\usepackage{amsfonts}
\usepackage{amscd} 

\usepackage{xcolor}

\usepackage{tikz}
\usetikzlibrary{arrows,automata} 
\usepackage{istgame} 

\tikzstyle{rect} = [rectangle, rounded corners, minimum width=1cm, minimum height=1cm,text centered, draw=black ]
\tikzstyle{arrow} = [thick,->,>=stealth]
\usetikzlibrary{decorations.markings,arrows.meta,bending}
\tikzset{->-/.style={decoration={markings, mark=at position #1 with {\arrow[line width=2pt]{>}}},postaction={decorate}}}

\usepackage{subcaption}

\usepackage{pdfpages} 

\usepackage{hyperref} 

\newcommand{\EE}{\mathbb{E}}

\newcommand{\PP}{\mathbb{P}}
\newcommand{\RR}{\mathbb{R}}
\newcommand{\NN}{\mathbb{N}}
\newcommand{\1}{\mathds{1}}

\newcommand{\defas}{\coloneqq}
\newcommand{\safed}{\eqqcolon}

\newcommand{\until}{\, .. \,}

\usepackage[normalem]{ulem} 

\usepackage{multirow} 
\usepackage{hhline} 

\usepackage{thmtools, thm-restate}

\newif\ifhideproofs
\ifhideproofs
  \usepackage{environ}
  \NewEnviron{hide}{}

\fi

\usepackage{todonotes} 

\usepackage{cleveref} 

\theoremstyle{plain} 
\newtheorem{@theorem}{Theorem}[section]
\newenvironment{theorem}{\begin{@theorem}}{\end{@theorem}}

\newtheorem{lemma}[@theorem]{Lemma}

\theoremstyle{definition} 

\theoremstyle{remark} 
\newtheorem{Remark}[@theorem]{Remark}

\begin{document}

\title{
	Repeated Prophet Inequality with Near-optimal Bounds\iftoggle{SODA}{\thanks{Full version with detailed proofs attached as Appendix.}}
}

\author{
	Krishnendu Chatterjee\thanks{Institute of Science and Technology Austria, Klosterneuburg, Austria.}
	\and
	Mona Mohammadi\thanks{Institute of Science and Technology Austria, Klosterneuburg, Austria.}
	\and
	Raimundo Saona\thanks{Institute of Science and Technology Austria, Klosterneuburg, Austria.}
}

\date{\today}

\maketitle
\thispagestyle{empty}

\newcommand{\ALG}{\text{ALG}}
\newcommand{\SEC}{\text{SEC}}
\newcommand{\WAI}{\text{WAI}}
\newcommand{\RPI}{\text{RPI}}
\newcommand{\TPS}{\text{TPS}}
\newcommand{\SOP}{\text{SOP}}
\newcommand{\sALG}{\textbf{ALG}}
\newcommand{\sSEC}{\textbf{SEC}}
\newcommand{\sWAI}{\textbf{WAI}}
\newcommand{\sRPI}{\textbf{RPI}}
\newcommand{\sTPS}{\textbf{TPS}}
\newcommand{\sSOP}{\textbf{SOP}}
\newcommand{\perf}{\text{perf}}

\begin{abstract}
    In modern sample-driven Prophet Inequality, an adversary chooses a sequence of $n$ items with values $v_1, v_2,  \ldots, v_n$ to be presented to a decision maker (DM).
    The process follows in two phases. In the first phase (sampling phase), some items, possibly selected at random, are revealed to the DM, but she can never accept them. In the second phase,
    the DM is presented with the other items in a random order and online fashion. 
    For each item, she must make an irrevocable decision to either accept the item and stop the process 
    or reject the item forever and proceed to the next item. 
    The goal of the DM is to maximize the expected value as compared to a Prophet (or offline algorithm) that has access to
    all information.
    In this setting, the sampling phase has no cost and is not part of the optimization process. However, in many scenarios, the samples are obtained as part of the decision-making 
    process.

    We model this aspect as a two-phase Prophet Inequality where an adversary chooses a sequence of $2n$ items with values $v_1, v_2, \ldots, v_{2n}$ and
    the items are randomly ordered. Finally, there are two phases of the Prophet Inequality problem with the first $n$-items and the rest of the items, respectively. 
    We show that some basic algorithms achieve a ratio of at most $0.450$. 
    We present an algorithm that achieves a ratio of at least $0.495$. 
    Finally, we show that for every algorithm the ratio it can achieve is at most $0.502$.
    Hence our algorithm is near-optimal.
\end{abstract}

\smallskip
\noindent\emph{Keywords:} Online algorithms; Prophet Inequality; Multi-choice.


\newpage
\pagenumbering{arabic}

\iftoggle{SODA}
{
    \vspace{-1em}
    \section{Introduction} 
    \label{Section: Introduction}
    \vspace{-0.5em}
}{
    \section{Introduction} 
    \label{Section: Introduction}
}

\smallskip\noindent{\em Online decision-making models.}
In computer science, online decision-making is a fundamental problem, and the two classical models are {\em competitive analysis} and
{\em optimal stopping}~\cite{borodin2005OnlineComputationCompetitive, shiryaev2008OptimalStoppingRules}.
In competitive analysis, the input is generated by an adversary for the worst-case analysis.
In contrast, in classical approaches for optimal stopping full knowledge about the distribution of the input is assumed.
The classical model for optimal stopping is as follows: $n$ items with values $v_1, v_2, \ldots, v_n$ are generated from 
a known distribution, and presented to the decision maker (DM) in an online fashion. 
For each item, the DM makes an irrevocable decision either to accept the item and stop, or reject the item and can never go back to the rejected item.
The objective is to maximize the expected payoff as compared to a Prophet who knows the values in advance (i.e., an offline algorithm). 
This is referred to as the i.i.d. Prophet Inequality~\cite{kertz1986StopRuleSupremum, correa2017PostedPriceMechanisms}, and if each item value is obtained from a different distribution,
then it is general Prophet Inequality~\cite{gilbert1966RecognizingMaximumSequence, krengel1977SemiamartsFiniteValues, krengel1978SemiamartsAmartsProcesses}.
If the goal is to maximize the probability to obtain the best item, then the problem is called Secretary Problem~\cite{ferguson1989WhoSolvedSecretary} (although in this case no knowledge of the distribution is usually assumed).
In general Prophet Inequality, the distribution for each item is different yet known, but there are three possibilities for the order
of presentation of items: (a)~adversarial (worst-case order); (b)~free (best-case order chosen by the DM); and (c)~random. 
For Prophet Inequality, 
(a)~if the order of the presentation is adversarial, then the optimal ratio is~$1/2$, established in the celebrated 
result~\cite{krengel1977SemiamartsFiniteValues, krengel1978SemiamartsAmartsProcesses};
(b)~if the order of presentation is free, then the optimal ratio is at least $0.725$~\cite{peng2022OrderSelectionProphet}; and
(c)~if the order of presentation is random, then the optimal ratio is at least $0.699$~\cite{correa2021ProphetSecretaryBlind}. 
Lastly, if the distribution for each item is unknown, we have the following results:
(a)~if the order of the presentation is adversarial, then the optimal ratio is~$1/n$, established also in the celebrated result~\cite{krengel1977SemiamartsFiniteValues, krengel1978SemiamartsAmartsProcesses};
(b)~if the order of presentation is free, then the optimal ratio is the same as that under a random order since there is no prior information; and
(c)~if the order of presentation is random, then the optimal ratio is $1/e \approx 0.367$~\cite{gilbert1966RecognizingMaximumSequence} as the problem is equivalent to the Secretary Problem. 

\smallskip\noindent{\em Motivation.} 
The assumption of worst-case inputs generated by an adversary in competitive analysis is quite strong in many real-world problems, such as,
in e-commerce platforms and online auctions, where the input is not generated adversarially. 
Here the optimal stopping problems are more suited.
Hence the study of optimal stopping problems has been an active research area, both from a theoretical perspective~\cite{hill1992SurveyProphetInequalities}
as well as practical perspective~\cite{einav2018AuctionsPostedPrices, surowiecki2011GoingGoingGone}. 
Moreover, the Prophet Inequality is closely related to Posted-Price Mechanisms (PPMs)~\cite{hajiaghayi2007AutomatedOnlineMechanism, chawla2010MultiparameterMechanismDesign, correa2019PricingProphetsBack}.
PPMs are an attractive alternative to implementing auctions and are usually used in online sales~\cite{chawla2010MultiparameterMechanismDesign, einav2018AuctionsPostedPrices}. 
Since they are sub-optimal, it is important to know the ratio between PPMs and the optimal auction (Myerson's auction). 
This ratio can be studied through Prophet Inequality~\cite{hajiaghayi2007AutomatedOnlineMechanism, chawla2010MultiparameterMechanismDesign, correa2019PricingProphetsBack}.

\smallskip\noindent{\em New models for Prophet Inequality.}
In contrast to the strong pessimistic assumption of adversarial inputs of competitive analysis, the classical optimal stopping models 
consider a strong optimistic assumption of full-distributional knowledge.
This strong assumption often does not hold in real-world scenarios, and the seminal work of Azar et. al.~\cite{azar2014ProphetInequalitiesLimited}
introduces the model of optimal stopping with limited information on the distribution.
The main idea is as follows. First, an adversary chooses $n$ items with values $v_1, v_2, \ldots, v_n$.
Then, the procedure has two phases:
(1)~{\em Phase~1: The sampling phase.} The DM is presented with a randomly chosen number of items, but she cannot accept them. 
(2)~{\em Phase~2: The optimal stopping phase.} The rest of the items are presented in a uniformely random order 
in an online fashion as the classical Prophet Inequality.
This model captures that the knowledge of the distributions is partially known and obtained by access to samples
or historical data. 
This line of research is an active topic with several follow-up works, for example, 
Correa et. al.~\cite{correa2020SampledrivenOptimalStopping} establishes the optimal ratio for a random number 
of sample values; and Kaplan et. al.~\cite{kaplan2019CompetitiveAnalysisSample} presents a sub-optimal ratio for a 
fixed number of sample values.
Moreover, the work of~\cite{correa2020SampledrivenOptimalStopping} also establishes that the optimal ratio for 
a fixed number of samples and random number of samples coincide in the limit as $n$ goes to $\infty$.

\smallskip\noindent{\em Our model.} 
In the above works, the samples obtained in the sampling phase play the role of historical data.
Moreover, in the above works, the sampling phase has no cost, i.e., the decision-making process is separate from
the sampling phase.
However, in many real-world scenarios, the data is obtained as part of the decision-making process. 
We present a simple adaptation of the above model to capture that the sampling phase is similar to the decision 
making process.
In our model, an adversary chooses $2n$ items with values $v_1, v_2, \ldots, v_{2n}$ and a random order is applied
to these values. Then, there are two phases of Prophet Inequality:
(1)~{\em Phase~1.} The DM is presented with the first $n$ items (according to the random order) in an online fashion 
and must make an irrevocable decision, and once a value is accepted the phase stops. 
(2)~{\em Phase~2.} The DM is presented with the last $n$ items (according to the random order) in an online 
fashion and must make an irrevocable decision at each step.
The payoff is the sum of values obtained in the two phases. 
This model is motivated by the fact that the samples are obtained as part of the decision-making process.

\smallskip\noindent{\em Our contribution.}
First, note that in our setting if there is only one phase, then the optimal ratio is $1/e \approx 0.367$. Indeed, the problem would be equivalent to the Secretary Problem and the optimal ratio was derived by~\cite{gilbert1966RecognizingMaximumSequence}. For the two-phase process we establish the following results:
\begin{compactitem}
    \item For basic algorithms, that either treat the two phases independently or there is a single observation period, the optimal ratio is at most $0.368$ and $0.450$, respectively.
    \item We present an algorithm that achieves a ratio of at least $0.495$.
    \item For every algorithm the ratio is at most $0.502$.
\end{compactitem}
In other words, the bound for our algorithm and the upper bound shows that our algorithm is near-optimal.

Finally, we also consider the extension where the payoff is a convex combination of the two phases, i.e., the sum of $(1 - \lambda)$ times Phase~1 and $\lambda$ times Phase~2, 
where $\lambda \in [0, 1]$.
This provides relative importance to the two phases.
We generalize our algorithm to provide a lower bound for any convex combination. Also, we generalize our upper bound. Particular values of $\lambda$ recover previous results. For example, $\lambda = 0$ models the situation where the DM is indifferent about Phase~2, recovering the classical single-choice Prophet Inequality with no prior information.
The case of $\lambda = 1$ models the situation where the DM is indifferent about Phase~1, recovering the Sample-driven Prophet Inequality where half of the items are sampled~\cite{correa2020SampledrivenOptimalStopping, kaplan2019CompetitiveAnalysisSample}. The optimal performance for this problem is approximately $0.671$, due to~\cite{correa2020SampledrivenOptimalStopping}. In conclusion, this recovers a known result and shows that we present a unifying approach.

\iftoggle{SODA}
{
    \smallskip
    \noindent\textbf{Further related works.}
}{
    \subsection{Further related works}
}
The Secretary Problem has been studied since 1950, and the optimal ratio is $1/e \approx 0.367$, see~\cite{ferguson1989WhoSolvedSecretary} for a survey of the Secretary Problem and its variants. 
The multi-choice Secretary Problem considers the selection of more than one item and has been studied in~\cite{gilbert1966RecognizingMaximumSequence, assaf2000SimpleRatioProphet, kleinberg2005MultiplechoiceSecretaryAlgorithm}.
Similarly, the multi-choice Prophet Inequality is an active research area. See for example~\cite{hajiaghayi2007AutomatedOnlineMechanism, alaei2014BayesianCombinatorialAuctions}.
In our setting, the DM can accept two values which must belong to different phases.
Although the models are similar, there is a crucial difference: both the DM and the Prophet are constrained to accept at most one item in 
each Phase, whereas previous models do not have this constraint.  
For example, in the 2-choice Prophet Inequality, if there is full knowledge of the distributions of the values, the work of Alaei~\cite[Section 4]{alaei2014BayesianCombinatorialAuctions} shows a lower bound of $1 - 1 / \sqrt{5} \approx 0.5527$.
In contrast, in our model (where there is no prior knowledge of the values and the phase constraint is present), we present a lower bound (through our algorithm) of~$0.495$ and an upper bound of~$0.502$.
This shows that these models, though similar, have quite different guarantees. Modern sample-driven Secretary Problem has been studied, for example, in~\cite{kaplan2019CompetitiveAnalysisSample, correa2021SecretaryProblemIndependent}.

\iftoggle{SODA}
{
    \vspace{-1em}
    \section{Preliminaries}
    \label{Section: Preliminaries}
    \vspace{-0.5em}
}{
    \section{Preliminaries}
    \label{Section: Preliminaries}
}

\subsection{Model}

We denote the set of natural numbers by $\NN$, the discrete interval containing only natural numbers between $a$ and $b$ by $[a \until b]$, i.e. $\{a, a+1, \ldots, b\} = [a, b] \cap \NN$, where $0 \le a \le b$.

\smallskip
\noindent\textbf{Repeated Prophet Inequality.} We consider the following optimal stopping problem. An adversary chooses a list of $2n$ items with associated real values $v_1 \ge v_2 \ge \ldots \ge v_{2n} \ge 0$ and nature chooses a random permutation $\sigma \colon [1 \until 2n] \to [1 \until 2n]$. Then, the process follows in two phases.
\begin{compactitem}
	\item \textbf{Phase~1.} The first $n$ items are presented to the DM one by one according to $\sigma$, i.e., $v_{\sigma(1)}, \allowbreak v_{\sigma(2)}, \allowbreak \ldots, \allowbreak v_{\sigma(n)}$. When the $i$-th item arrives, the DM does not have access to the numerical value $v_{\sigma(i)}$ but can compare it with the previous items, i.e. $v_{\sigma(1)}, v_{\sigma(2)}, \ldots, v_{\sigma(i - 1)}$. In other words, she knows its relative ranking within the items observed so far. After the arrival of the $i$-th item, the DM must make an irrevocable decision to either accept the item and stop this phase (no more items from this phase will be revealed) or reject the item forever and proceed to the next item, if any. When accepting an item, its value is not revealed to the DM.

	\item \textbf{Phase~2.} The last $n$ items are presented to the DM one by one according to $\sigma$, i.e., $v_{\sigma(n + 1)}, \allowbreak v_{\sigma(n + 2)}, \allowbreak \ldots, \allowbreak v_{\sigma(2n)}$. When the $i$-th item of this phase arrives, the DM does not have access to the numerical value $v_{\sigma(n + i)}$ but can compare it with the previous items from this phase and all previously seen items from Phase~1. In other words, she knows its relative ranking within the items observed so far. After the arrival of the $i$-th item, the DM must make an irrevocable decision to either accept the item and stop this phase, and therefore the process, or reject the item forever and proceed to the next item.
\end{compactitem}
Since items may have equal values, we assume that there is an arbitrary tie-breaking rule that is consistent with the relative ranks revealed and publicly selected before the process starts.
In particular, only after observing all items, the DM can be sure of the absolute ranking of each item.
If in any of the two phases the DM does not accept an item, then she gets an item with a value of zero in the respective phase.
Her goal is to maximize the expectation of what she gets, i.e., the sum of the values of the items she accepted.

\smallskip
\noindent\textbf{Value-knowledge.} Upon an instance $\bm{v}$, the DM may or may not know the values $v_1, v_2, \ldots, v_{2n}$ in advance. This defines two variants of the Repeated Prophet Inequality. If the DM knows these values in advance, we call it the \emph{values observed} variant. If the DM does not know the values in advance, we call it the \emph{value-oblivious} variant.

\smallskip
\noindent\textbf{Algorithm.} In the Repeated Prophet Inequality, the number of items is fixed and known by the DM. Then, an \emph{algorithm} for the DM describes how to act while processing $2n$ items and is denoted by $\ALG_n$. By following the algorithm $\ALG_n$ on an instance $\bm{v}$, the DM obtains the sum of two random values, denoted by $\ALG_n(\bm{v})$ (or simply $\ALG_n$ when the instance is clear from context). Abusing of notation, we call \emph{algorithm} to a family of algorithms, one for every instance size, formally $\sALG = (\ALG_n)_{n \in \NN}$.

\smallskip
\noindent\textbf{Limit performance.} Following a competitive analysis perspective, we are interested in the worst case ratio (over all possible instances) between the expectation of what the DM accepts and the expectation of what a Prophet, who always accepts the two items with the largest values, obtains. This corresponds to the performance of an optimal offline algorithm. 
The \emph{limit performance} of $\sALG$ is the worst case ratio it guarantees for instances arbitrarily large. Formally, the \emph{limit performance} of $\sALG$ is
\[
	\perf(\sALG) 
    \defas \liminf_{n \to \infty} \inf_{\bm{v}}
	   \frac{ \EE(\ALG_n({\bm{v}})) }{ \EE \left(\max \left\{ v_{\sigma(i)} : i \in [1 \until n] \right\} + \max \left\{ v_{\sigma(i)} : i \in [(n + 1) \until 2n] \right\} \right) } \,.
\]
We are interested in the best limit performance possible, i.e., abusing of notation,
\[
	\perf \defas \sup_{\sALG} \perf(\sALG) \,.
\]

\smallskip
\noindent\textbf{Record.} When the DM is faced with a new item whose value is larger than everything she has observed so far, we call this item a \emph{record}. Records are of particular interest in many optimal stopping time problems and are used to define many algorithms later on.

\smallskip
\noindent\textbf{Repeated Secretary Problem.} As a variant of the Repeated Prophet Inequality, we also consider a different but closely related goal for the DM. In this optimal stopping problem, the goal of the DM is to maximize the probability of accepting the maximum value in the instance (denote by $v_1$), as opposed to maximizing the expected value of the sum of her choices. Note that, unlike the classical Secretary problem, the DM has two opportunities to pick the maximum value.

\subsection{Our Contributions}

\begin{lemma}
	\label{Result: Observe-exploit basic algorithm upperbound}
	Algorithms that have a single observation period and pick any value greater than the items in the observation period have a limit performance of at most $0.450$.
\end{lemma}

\begin{lemma}
	\label{Result: Independent-phases Basic algorithm upperbound}
	Algorithms that do not compare items from the second phase to items of the first phase have a limit performance of at most $1/e \approx 0.367$.
\end{lemma}

In \Cref{Section: Our Algorithm}, we define our algorithm that has the following property.

\begin{theorem}
	\label{Result: Algorithm lowerbound}
	Our algorithm has a limit performance of at least $0.495$.
\end{theorem}

\begin{theorem}
	\label{Result: Algorithms upperbound}
	For every algorithm, the optimal limit performance is at most $0.502$.
\end{theorem}

\begin{Remark}
	Our algorithm of \Cref{Result: Algorithm lowerbound} is value-oblivious. Moreover, all upper bounds (\Cref{Result: Observe-exploit basic algorithm upperbound}, \Cref{Result: Independent-phases Basic algorithm upperbound} and \Cref{Result: Algorithms upperbound}) apply to the value observed variant of the problem. 
    Therefore, the above bounds apply to both value observed and value oblivious variants of the problem.
\end{Remark}

\subsection{Selected previous results}

We recall some foundational results due to~\cite{gilbert1966RecognizingMaximumSequence}, which we use in the sequel. 

\smallskip
\noindent \emph{Secretary strategy.} Consider the classical Secretary Problem. Then, for $x \in [0, 1]$, denote $\sSEC[x]$ the strategy that proceeds as follows. For an instance of size $n$, the strategy $\SEC[x]_{n}$ observes the first $\lceil xn \rceil$ items and then accepts any record that appears.

\begin{lemma}[Acceptance probability of Secretary algorithms~\cite{gilbert1966RecognizingMaximumSequence}]
\label{Result: Acceptance probability of Secretary algorithms}
    For all $x \in [0, 1]$, we have
    \[
        \lim_{n \to \infty} \PP(\SEC[x]_{n} \text{ accepts } v_1) = -x \ln(x) \,.
    \]
\end{lemma}

\begin{lemma}[Optimal strategy in the Secretary Problem~\cite{gilbert1966RecognizingMaximumSequence}]
\label{Result: Optimal strategy in the Secretary Problem}
    Let a fixed proportion $\rho \in [0, 1]$. Consider the $1$-choice Secretary Problem with $n$ items where the DM is restricted to observe but not accept the first $\lceil n \rho \rceil$ items revealed. Then, an asymptotically optimal algorithm is to observe but not accept the first $\max \{ \lceil n \rho \rceil , \lceil n / e \rceil \}$ items and then accept any record that appears.
\end{lemma}

\noindent\textbf{Overview of the paper.}
The rest of the paper is organized as follows. \Cref{Section: Basic algorithms} shows that basic algorithms achieve a performance of at most $0.450$. \Cref{Section: Our Algorithm} presents an algorithm that achieves a performance of at least $0.495$. \Cref{Section: Upper bound} shows an upper bound of $0.502$ that applies to any possible algorithm. \Cref{Section: Convex Combination of Phases} extends our results for the case that the goal of the DM is a convex combination of what she obtains during Phase~1 and Phase~2.

\iftoggle{SODA}
{
    \vspace{-1em}
    \section{Performance of Basic Algorithms }
    \label{Section: Basic algorithms}
    \vspace{-0.5em}
}{
    \section{Performance of Basic Algorithms }
    \label{Section: Basic algorithms}
}

In this section, we analyze two basic algorithms. First, we consider algorithms with a single observation period. Then, we consider algorithms that treat both phases independently.  \Cref{Result: Observe-exploit basic algorithm upperbound} and \Cref{Result: Independent-phases Basic algorithm upperbound} state that they have an upper bound of $0.450$ and $0.368$, respectively.

\iftoggle{SODA}
{
    \smallskip
    \noindent\textbf{Single observation period.}
    We analyze}{
    \subsection{Single observation period}

    In this section, we analyze}
algorithms that act in two stages. For $x \in [0, 1]$, we denote the algorithm $\sSOP[x]$. Upon an instance with $2n$ items, $\SOP[x]_n$ proceeds as follows. The first $\lceil xn \rceil$ items are observed but not accepted. This constitutes the observation stage. Then, the online stage starts. During Phase~1, any new item is compared only with the items from the observation stage and chosen if its value is larger than all of them. During Phase~2, $\SOP[x]_n$ proceeds similarly, comparing only with items from the observation stage. In particular, an item accepted during Phase~2 may not be larger than an item accepted during Phase~1 because values of Phase~2 are compared only to the observation period. \Cref{Result: Observe-exploit basic algorithm upperbound} states that, for all $x \in [0, 1]$, the algorithm $\sSOP[x]$ has a limit performance of at most $0.450$.

\emph{Overview of the proof.} We start by proving a general upper bound that applies to all algorithms for the Repeated Prophet Inequality. The upper bound states that the limit performance of any algorithm is at most the probability of accepting the item with the largest value, i.e. $v_1$. Then, we deduce the exact probability for the family of algorithms $(\sSOP[x])_{x \in [0, 1]}$. Finally, optimizing over $x$ the result follows.

\begin{lemma}[Choosing the maximum upperbound]
\label{Result: Choosing the maximum upperbound}
	Consider the Repeated Prophet Inequality, with known or unknown values. Then, the limit performance of any algorithm is bounded by its probability of accepting the maximum value in arbitrarily big instances. Formally, for any algorithm $\sALG$,
	\[
		\perf(\sALG) \le \liminf_{n \to \infty} \PP(\ALG_n \text{ accepts the maximum value}) \,.
	\]
\end{lemma}

\iftoggle{SODA}
{
    \begin{proof}[Sketch proof]
        Consider the instance where the maximum value is one and all others are zero, i.e. ${\bm v} = (1, 0, 0, \ldots, 0) \in \RR^{2n}$. Then, the expectation of the DM is one only if she accepts $v_1$. On the other hand, the expectation of the prophet is always exactly one. Therefore, since this is a particular instance, the performance of $\sALG$ is bounded by the corresponding probability of choosing the item with maximum value.
    \end{proof}
}{
    \begin{proof}
    	Fix $n \in \NN$. Consider the instance where the maximum value is one and all others are zero, i.e. ${\bm v} = (1, 0, 0, \ldots, 0) \in \RR^{2n}$. Then, the prophet always obtains $1$, while the DM obtains $1$ only if she accepts the maximum value $v_1$. Therefore, for any strategy $\ALG_n$, we have that
    	\[
    		\frac{ \EE(\ALG_n({\bm{v}})) }{ \EE \left(\max \left\{ v_{\sigma(i)} : i \in [1 \until n] \right\} + \max \left\{ v_{\sigma(i)} : i 	\in [(n + 1) \until 2n] \right\} \right) }
    		= \PP(\ALG_n \text{ accepts the maximum value}) \,.
    	\]
    	Taking the limit of $n$ growing to $\infty$, we conclude that, for any algorithm $\sALG$,
    	\[
    		\perf(\sALG) \le \liminf_{n \to \infty} \PP(\ALG_n \text{ accepts the maximum value}) \,.
    	\]
    \end{proof}
}

\iftoggle{SODA}
{
    \begin{proof}[Sketch proof of \Cref{Result: Observe-exploit basic algorithm upperbound}]
        Let $x \in [0, 1]$. We prove that $\perf(\sSOP[x]) \le 0.450$. By \Cref{Result: Choosing the maximum upperbound}, we only need to bound the probability that $\sSOP[x]$ accepts $v_1$. During Phase~1, $\sSOP[x]$ acts as a Secretary algorithm where $\lceil xn \rceil$ items are observed and there are $n$ items in total. During Phase~2, $\sSOP[x]$ acts as a Secretary algorithm where $\lceil xn \rceil$ items are observed and there are $( \lceil xn \rceil + n)$ items in total. Therefore, by \Cref{Result: Acceptance probability of Secretary algorithms}, the corresponding probabilities are $-x \ln(x) / 2$ and $\ln \left( x / (x + 1) \right) / 2$. Optimizing over $x \in [0, 1]$ yields the result.
    \end{proof}
}{
    We now turn to prove \Cref{Result: Observe-exploit basic algorithm upperbound}.

    \begin{proof}[Proof of \Cref{Result: Observe-exploit basic algorithm upperbound}]
    	Fix $x \in [0, 1]$ and $n \in \NN$. Denote $E_1$ and $E_2$ the events that $\SOP[x]_n$ accepts the item with the largest value during Phase~1 and Phase~2 respectively. Then,
    	\begin{align*}
    		\PP(E_1)
    		 & = \frac{1}{2} \PP( E_1 \mid \sigma^{-1}(1) \le n)
    		 & (v_1 \text{is in Phase~1})                                       \\
    		 & = \frac{1}{2} \sum_{i = 1}^n \PP( E_1 \mid \sigma^{-1}(1) = i) \cdot \PP( \sigma^{-1}(1) = i \mid \sigma^{-1}(1) \le n)
    		 & ( \text{position of } v_1 )                                   \\
    		 & = \frac{1}{2} \sum_{i = \lceil xn \rceil + 1}^n \PP( E_1 \mid \sigma^{-1}(1) = i) \cdot \PP( \sigma^{-1}(1) = i \mid \sigma^{-1}(1) \le n)
    		 & ( \SOP[x]_n \text{ rejects } \lceil xn \rceil \text{ items})              \\
    		 & = \frac{1}{2} \sum_{i = \lceil xn \rceil + 1}^n \frac{\lceil xn \rceil}{i - 1} \cdot \frac{1}{n}   \\
    		 & \xrightarrow[n \to \infty]{} \frac{ 1 }{ 2 } x \int_x^1 \frac{1}{u} du            \\
    		 & = \frac{ 1 }{ 2 } x \left. \ln \left( u \right) \right\rvert_x^1               \\
    		 & = - \frac{ 1 }{ 2 } x \ln(x) \,,
    	\end{align*}
        where we used that, given that $v_1$ appears in position $i$, the probability of accepting $v_1$ equals the probability that the item with the largest value among the first $(i - 1)$ items is revealed at a position smaller or equal to $\lceil xn \rceil$. Formally, 
        \[
            \PP( E_1 \mid \sigma^{-1}(1) = i) = \PP( \max \{ v_{\sigma(j)} : j \le \lceil xn \rceil \} = \max \{ v_{\sigma(j)} : j \le (i - 1) \} ) = \frac{\lceil xn \rceil}{i - 1} \,.
        \]
    
    	On the other hand, in the second phase we may accept any of the $n$ items, provided they are larger than the observation period. Therefore,
    	\begin{align*}
    		\PP(E_2)
    		 & = \frac{1}{2} \PP( E_2 \mid \sigma^{-1}(1) > n) \\
             & = \frac{1}{2} \sum_{i = 1}^n \PP( E_2 \mid \sigma^{-1}(1) = n + i) 
                \cdot \PP(\sigma^{-1}(1) = n + i \mid \sigma^{-1}(1) > n)            \\
    		 & = \frac{1}{2} \sum_{i = 1}^n \frac{\lceil xn \rceil}{ \lceil xn \rceil + i - 1 } \cdot \frac{1}{n} \\
    		 & \xrightarrow[n \to \infty]{} \frac{ 1 }{ 2 } \int_0^1 \frac{x}{x + u} du      \\
    		 & = \frac{ 1 }{ 2 } x \left. \ln \left( x + u \right) \right\rvert_0^1       \\
    		 & = - \frac{ 1 }{ 2 } x \ln \left( \frac{x}{x + 1} \right) \,.
    	\end{align*}
    
    	Finally, adding up this probabilities, we have that
    	\begin{align*}
    		\lim_{n \to \infty} \PP(\SOP[x]_n \text{ accepts the maximum value})
    		 & = \lim_{n \to \infty} \PP(E_1) + \PP(E_2)                   \\
    		 & = - \frac{1}{2} x \ln \left( \frac{x^2}{x + 1} \right) \,.
    	\end{align*}
    	This function is strictly concave and standard optimization techniques yield that the maximum is less than $0.450$ at $x^* \approx 0.545$.
    \end{proof}
}

\iftoggle{SODA}
{
    \smallskip
    \noindent\textbf{Treat phases separately.} 
    We analyze}
{
    \subsection{Treat phases separately}

    In this section, we analyze}
algorithms that treat the first $n$ items and the last $n$ items as separate and independent single-selection optimal stopping problems. In particular, the DM does not compare items observed in Phase~2 to items from Phase~1. \Cref{Result: Independent-phases Basic algorithm upperbound} states that these algorithms have a limit performance of at most $1/e \le 0.368$.

\iftoggle{SODA}
{
    \begin{proof}[Sketch proof of \Cref{Result: Independent-phases Basic algorithm upperbound}]
        Consider the upper bound given by \Cref{Result: Choosing the maximum upperbound}. For algorithms that treat phases separately, each phase reduces to the classical Secretary Problem. Indeed, in each phase, the algorithm knows that only picking the maximum value counts and knows that items are shown in random order. By \Cref{Result: Optimal strategy in the Secretary Problem}, the optimal strategy is to observe a portion $1/e$ of the items in each round. For this algorithm, the DM accepts the maximum value with probability $1/e$ in each phase separately and therefore the same bound holds for both phases when put together.
    \end{proof}
}{
    \begin{proof}[Proof of \Cref{Result: Independent-phases Basic algorithm upperbound}]
        By \Cref{Result: Choosing the maximum upperbound}, it is enough to prove that, for algorithms that treat phases separately, the probability of accepting the item with maximum value is bounded by $1/e$. Since there are no consequences in Phase~2 from the decisions taken in phase 1, the DM can optimize her strategy in each phase separately.
        
        In each phase, the strategy that maximizes the probability of accepting the item with the largest value is the strategy that maximizes the probability of accepting the largest value in the respective phase. The only information the DM has available is that values are presented in an online fashion, in a random order and, after an item is revealed, she has access to its relative ranking. This is exactly the setting of Secretary Problem~\cite{ferguson1989WhoSolvedSecretary}. Therefore, the maximum probability of accepting the maximum in each phase is $1/e$. The strategy to achieve this bound consists of observing the first $1/e$ portion of the items and then accepting any record that appears. Since this is the optimal strategy for each phase, the algorithm that maximizes the probability of accepting the item with the largest value repeats this strategy in each phase. Formally, for all algorithms that treat phases separately $\sTPS$,
        \[
            \perf(\sTPS) \le \lim_{n \to \infty} \PP(\TPS_n \text{ accepts } v_1) \le \frac{1}{e} \le 0.368 \,,
        \]
        which proves the claim.
    \end{proof}
}

\iftoggle{SODA}
{
    \vspace{-1em}
    \section{Our Algorithm}
    \label{Section: Our Algorithm}
    \vspace{-0.5em}
}{
    \section{Our Algorithm}
    \label{Section: Our Algorithm}
}

In this section, we present an algorithm whose limit performance is at least $0.495$. Our algorithm extends the work of~\cite{correa2020SampledrivenOptimalStopping} which was designed for the Sample-driven Prophet Inequality problem, where samples are given and do not take part in the optimization process. We start by recalling some notions from~\cite{correa2020SampledrivenOptimalStopping} and explain the issues we need to overcome. Finally, we present our algorithm and analyze its limit performance.

\smallskip
\noindent \emph{Random Sample-driven Prophet Inequality.} Consider the following single selection optimal stopping problems with random arrival order and independent sampling of items. An adversary chooses a list of $n$ items with associated real values $v_1 \ge v_2 \ge \ldots \ge v_{n} \ge 0$. The values are not presented to the DM. In the first phase (sampling phase), each item is selected as a sample independently with a fixed probability $p \in [0, 1]$. Samples are shown to the DM. In the second phase (online phase), the items that were not selected as samples are shown in an online fashion and according to a random order. When an item arrives, the DM does not have access to its value but can compare it with all previously seen items, i.e. the DM has access to its relative ranking. After an item arrives, the DM must make an irrevocable decision to either accept the item (and stop the process) or reject the item forever and proceed to the next item, if any. Since items may have equal values, we assume that there is an arbitrary tie-breaking rule that is consistent with the relative ranks revealed and publicly selected before the process starts. The goal of the DM is to maximize the expectation of the value of the item selected.

\smallskip
\noindent \emph{Main idea of~\cite{correa2020SampledrivenOptimalStopping}.} Based on the linear programming approach of~\cite{buchbinder2014SecretaryProblemsLinear}, and using mass moving arguments from optimal transport, the limit performance of the optimal strategy coincides with the solution of a certain real optimization problem. From the optimization problem, the general form of the optimal strategy is deduced. The optimal strategy proceeds as follows. There is a fixed infinite sequence of thresholds $t_1 \le t_2 \le \ldots$ all in $[0, 1]$ such that the optimal strategy accepts an item with relative ranking $i$ if it appears after a portion $t_i$ of items has been revealed. In particular, the optimal strategy is value-oblivious.

\smallskip
\noindent \emph{Results of~\cite{correa2020SampledrivenOptimalStopping}.} For each $p \in [0, 1]$, computing the limit performance of the optimal strategy, denoted $\alpha(p)$, amounts to solving an optimization problem with infinitely many variables. The numerical value of $\alpha(p)$ can be approximated by solving a pair of finite optimization problems that yield upper and lower bounds respectively. These finite optimization problems are mainly parametrized by the number of thresholds a strategy is allowed to use. Moreover, for some values of $p$, the value $\alpha(p)$ can be given exactly. For $p \le 1/e$, they prove that $\alpha(p) = (e (1 - p) )^{-1}$. Also, $\alpha(1) = \lim_{p \to 1} \alpha(p) \approx 0.745$, the same ratio one obtains in the i.i.d. Prophet Inequality. Another interesting value is $\alpha(1/2) \approx 0.671$, which corresponds to the single selection problem where the DM observes half of the items beforehand.

\smallskip
\noindent \emph{Issues in our setting.} The strategy of~\cite{correa2020SampledrivenOptimalStopping} is optimal for the single selection problem, but in our setting the DM is allowed to accept two items. Interpreting Phase~1 and Phase~2 as independent single selection problems will not lead to an optimal performance. Indeed, during Phase~1, since there are no samples, we would use the strategy of~\cite{correa2020SampledrivenOptimalStopping} for $p = 0$ (which coincides with the solution of the Secretary Problem) during Phase~1, i.e. observe but not accept the first $\lceil n / e \rceil$ elements and then accept any record that appears if any. Then, during Phase~2, interpret observed values as samples and apply the strategy of~\cite{correa2020SampledrivenOptimalStopping} for the corresponding value of $p$. This strategy is sub-optimal because one can increase the observation time to have more samples during Phase~2, see \Cref{Section: Our Algorithm} of Appendix for details. In conclusion, we derive that this is not an optimal strategy: one needs to consider both phases at the same time.

\smallskip
\noindent \emph{Our generalization.} We generalize the strategy of~\cite{correa2020SampledrivenOptimalStopping} and design a simple algorithm for the Repeated Prophet Inequality. For $x \in [0, 1]$, denote the algorithm $\sRPI[x]$ that proceeds as follows. 
\begin{compactitem}
	\item \textbf{Phase~1.} $\RPI[x]_{n}$ observes but does not accepts the first $\lceil xn \rceil$ items and then accepts any record that appears. 
    \item \textbf{Phase~2.} After observing during Phase~1 a total of $t$ items, $\RPI[x]_{n}$ proceeds as the optimal strategy of~\cite{correa2020SampledrivenOptimalStopping} with $p = t / (n + t)$, i.e. it uses a sequence of thresholds $t_1 \le t_2 \le \ldots$ to determine if an item should be accepted based only on its relative ranking and time of appearance.
\end{compactitem}
In particular, $\sRPI[x]$ is value-oblivious. Note that, if the size of the instance $n$ is large enough, then, conditioned on having observed $t$ items during Phase~1, the expectation of the item accepted is approximately $\alpha( t / (n + t))$ times that of the maximum of Phase~2. Therefore, we derive an expression for $\perf(\sRPI[x])$ that only depends on $x$ and $\alpha(\cdot)$. In particular, $\perf(\sRPI[x])$ is easy to evaluate numerically.

\smallskip
\noindent \emph{Overview of the proof.} To compute the limit performance of $\sRPI[x]$, we reduce the analysis to computing the limit performance of Phase~1 and Phase~2 separately. For Phase~1, by \Cref{Result: Acceptance probability of Secretary algorithms}, the limit performance is $-x \ln(x)$. For Phase~2, since we are computing the limit performance of just this phase, the items observed during Phase~1 can be interpreted as samples. Therefore, we can apply the strategy of~\cite{correa2020SampledrivenOptimalStopping}. The limit performance results in a certain expectation of $\alpha(p)$, where the value of $p$ depends on the (random) number of items observed during Phase~1, denoted by $T[x]$. The distribution of $T[x]$ is deduced and allows us to write the limit performance for Phase~2 in terms of $x$ and $\alpha(\cdot)$. Finally, we conclude the lower bound of $0.495$ stated in \Cref{Result: Algorithm lowerbound} by optimizing over $x \in [0, 1]$.

\iftoggle{SODA}
{

}{
    We start by showing how to express the performance of any algorithm in terms of its performance in each phase. Formally,
    \begin{lemma}[Equivalent form of limit performance]
    \label{Result: Equivalent form of limit performance}
        Consider any algorithm $\sALG$ for the Repeated Prophet Inequality. Then,
        \[
            \perf(\sALG) 
                = \liminf_{n \to \infty} \inf_{\bm{v}}
                    \frac{ \EE(\ALG_n^{(1)}({\bm{v}})) }{ 2 \EE \left(\max \left\{ v_{\sigma(i)} : i \in [1 \until n] \right\} \right) } + \frac{ \EE(\ALG_n^{(2)}({\bm{v}})) }{ 2 \EE \left(\max \left\{ v_{\sigma(i)} : i \in [1 \until n] \right\} \right) } \,,
        \]	
        where $\ALG_n^{(1)}$ and $\ALG_n^{(2)}$ denote the value accepted during Phase~1 and~2 respectively.
    \end{lemma}
    
    \begin{proof}
        Fix an arbitrary algorithm $\sALG$, an instance size $n \in \NN$ and an instance $\bm{v}$. Then, by definition, $\ALG_n({\bm{v}}) = \ALG_n^{(1)}({\bm{v}}) + \ALG_n^{(2)}({\bm{v}})$. Also, since $\sigma$ is a random order, we have that
        \[
            \EE \left(\max \left\{ v_{\sigma(i)} : i \in [1 \until n] \right\} + \max \left\{ v_{\sigma(i)} : i \in [(n + 1) \until 2n] \right\} \right)
                = 2 \EE \left(\max \left\{ v_{\sigma(i)} : i \in [1 \until n] \right\} \right) \,.
        \]
        Finally, we conclude that
        \begin{align*}
            \perf(\sALG) 
                &= \frac{ \EE(\ALG_n({\bm{v}})) }{ \EE \left(\max \left\{ v_{\sigma(i)} : i \in [1 \until n] \right\} + \max \left\{ v_{\sigma(i)} : i \in [(n + 1) \until 2n] \right\} \right) } \\
                &= \frac{ \EE \left( \ALG_n^{(1)}({\bm{v}}) \right) + \EE \left( \ALG_n^{(2)}({\bm{v}}) \right) }{ 2 \EE \left(\max \left\{ v_{\sigma(i)} : i \in [1 \until n] \right\} \right) } \\
                &= \frac{ \EE(\ALG_n^{(1)}({\bm{v}})) }{ 2 \EE \left(\max \left\{ v_{\sigma(i)} : i \in [1 \until n] \right\} \right) } + \frac{ \EE(\ALG_n^{(2)}({\bm{v}})) }{ 2 \EE \left(\max \left\{ v_{\sigma(i)} : i \in [1 \until n] \right\} \right) } \,,
        \end{align*}
        which proves the lemma.
    \end{proof}
    
    We continue by formalizing the distribution of the stopping time during Phase~1 of $\bm{RPI}[x]$, which corresponds to the number of items revealed to the DM before starting Phase~2. Fix $n \in \NN$. Intuitively, the DM observes exactly $t$ items during Phase~1 while using algorithm $\RPI[x]_n$ if the following conditions hold: (1)~the number $t$ is at least $(\lceil xn \rceil + 1)$ since the algorithm discards the first $\lceil xn \rceil$ items; (2)~no record has been revealed after $\lceil xn \rceil$ items were revealed and before $t$ items were revealed, since, otherwise, the algorithm would have stopped; and (3)~item $t$ is a record or $t= n$ since in both cases the algorithm does not observe more items in Phase~1. Formally, we have the following result.
    \begin{lemma}[Distribution of stopping time]
    \label{Result: Distribution of stopping time}
    	Let $x \in [0, 1]$. Fix $n \in \NN$. Denote $T[x]$ the number of items observed during Phase~1 by $\RPI[x]_n$. Then, for all $t \in \{1, 2, \ldots, n \}$,
    	\[
    		\PP(T[x] = t) = \frac{\lceil xn \rceil}{t (t - 1)} \1[t > \lceil xn \rceil] + \frac{\lceil xn \rceil}{n} \1[t = n] \,.
    	\]
    \end{lemma}
    
    \iftoggle{SODA}
    {
        \begin{proof}[Sketch proof]
            Let $x \in [0, 1]$. Fix $n \in \NN$. Then, $\RPI[x]_n$ observes $n$ items if no record is revealed after the first $\lceil xn \rceil$ items were revealed. On the other hand, $\RPI[x]_n$ accepts the item that $\sigma$ assigns position position $t$ if it is a record and $\sigma$ assigns the item with the second maximum value, among the first $t$ items, a position of at most $\lceil xn \rceil$. 
        \end{proof}
    }{
        \begin{proof}
            Consider $t \in [( \lceil xn \rceil + 1) \until (n - 1)]$. Then, $T[x] = t$ occurs only if the DM has observed a record for the first time after observing $\lceil xn \rceil$ items and that record is revealed at position $i$. Therefore, since $\sigma$ is uniformly random,
        	\begin{align*}
                \PP(T[x] = t)
                    &= \PP(v_{\sigma(t)} > \max \{ v_{\sigma(i)} : i < t \}, \max \{ v_{\sigma(i)} : i \le \lceil xn \rceil \} = \max \{ v_{\sigma(i)} : i < t \} ) \\
                    & = \frac{1}{t} \frac{\lceil xn \rceil}{t - 1} \,.
        	\end{align*}
        
        	On the other hand, $T[x] = n$ can also occur when $\sigma$ assigns the item with maximum value in Phase~1 as one of the first $\lceil xn \rceil$ items. In this case, the DM never encounters a record after the first $\lceil xn \rceil$ items. Therefore, we must also consider this probability for the case $t = n$. Formally,
            \[
                \PP(T[x] = n) = \frac{\lceil xn \rceil}{n (n - 1)} + \frac{\lceil xn \rceil}{n} \,,
            \]
            which proves the lemma.
        \end{proof}
    }
    
    We turn to compute the limit performance of $\sRPI[x]$.
    
    \begin{lemma}[{Limit performance of $\sRPI[x]$}]
    \label{Result: Limit performance of RPI[x]}
        Consider the Repeated Prophet Inequality. Let $x \in [0, 1]$. Then,
        \[
            \perf(\sRPI[x]) 
                = \frac{x}{2} \left( -\ln(x) + \int_{x}^{1} \alpha \left( \frac{u}{1 + u} \right) \frac{1}{u^2} du + \alpha \left( \frac{1}{2} \right) \right) \,.
        \] 
    \end{lemma}
    
    \iftoggle{SODA}
    {
        \begin{proof}[Sketch proof]
            Fix $x \in [0, 1]$. Let $n \in \NN$. Then, based on \Cref{Result: Equivalent form of limit performance}, we can analyze each phase in isolation. For Phase~1, by \Cref{Result: Acceptance probability of Secretary algorithms}, the algorithm $\RPI[x]_n$ accepts the item with maximum value in Phase~1 with probability at least $-x \ln(x)$. Since it does so in a value-oblivious manner, its expected value is at least $-x \ln(x)$ times $\EE \left(\max \left\{ v_{\sigma(i)} : i \in [1 \until n] \right\} \right)$. For Phase~2, we condition on the number of items observed during Phase~1, say $t$ items. Then, $\RPI[x]_n$ acts a the strategy of~\cite{correa2020SampledrivenOptimalStopping} where there are $t$ samples and $(n + t)$ total items. Fixing the proportion $u \defas t / n$, we have that the performance during Phase~2 converges to $\alpha(p = u / (1 + u))$. By \Cref{Result: Distribution of stopping time} and taking the limit as $n$ grows to $\infty$, we obtain the corresponding integral and the factor of $\alpha(1/2)$ (which corresponds to the case when the DM observed $n$ items during Phase~1). 
        \end{proof}
    }{
        \begin{proof}
            Fix $x \in [0, 1]$. Let $n \in \NN$ and consider an instance $\bm{v}$ of $2n$ items. During Phase~1, $\RPI[x]_n$ observes the first $\lceil xn \rceil$ elements and then chooses the first record that appears. Therefore, the probability of accepting the item with the maximum value between those in Phase~1 is equal to $-x \ln(x)$~\cite{gilbert1966RecognizingMaximumSequence}. Formally,
            \[
                \PP(\RPI[x]_n^{(1)}(\bm{v}) = \max \{ v_{\sigma(i)} : i \in [1 \until n] \}) \ge - x \ln(x) \,.
            \]
            Moreover, since $\RPI[x]_n$ is value-oblivious this inequality is translated into an expectation inequality. Formally,
            \[
                \EE(\RPI[x]_n^{(1)}(\bm{v})) \ge -x \ln(x) \EE(\max \{ v_{\sigma(i)} : i \in [1 \until n] \})) \,.
            \]
            This inequality is tight when we consider particular instances $\bm{v}$. 
            
            During Phase~2, by conditioning on the items observered during Phase~1, we have that
            \begin{align}
            \label{Eq: Second round decomposition}
                \begin{aligned}
                    \EE \left( \RPI_n^{(2)}[x] (\bm{v}) \right) 
                        &= \sum_{t = \lceil xn \rceil + 1}^{n} \EE \left( \RPI_n^{(2)}[x] (\bm{v}) \mid T[x] = t \right) \PP( T[x] = t ) \\
                        &= \sum_{t = \lceil xn \rceil + 1}^{n - 1} \EE \left( \RPI_n^{(2)}[x] (\bm{v}) \mid T[x] = t \right) \PP( T[x] = t ) \\
                        &\quad\quad + \EE \left( \RPI_n^{(2)}[x] (\bm{v}) \mid T[x] = n \right) \PP( T[x] = n ) \,.
                \end{aligned} 
            \end{align}
        
            Fixing $t/n = u \in [0, 1]$, we have that $t / (n + t)$ converges to $u / (1 + u)$. In particular, the items revealed during Phase~1 amount to a portion $u / (1 + u)$ of the total number of items if one were to consider Phase~2 as a single selection problem where Phase~1 consists of only samples. Therefore, 
            \[
                \lim_{n \to \infty} \frac{ \EE \left( \RPI_n^{(2)}[x] (\bm{v}) \mid T[x] = t \right) }{ \EE \left(\max \left\{ v_{\sigma(i)} : i \in [1 \until n] \right\} \right) }
                    = \alpha \left( p = \frac{u}{1 + u} \right) \,.
            \]
            In particular, this limit is indpendent of $\bm{v}$. This is expected since the algorithm is value-oblivious. 
        
            Then, replacing the value of the stopping time probabilities given by \Cref{Result: Distribution of stopping time} in \eqref{Eq: Second round decomposition}, we get that
            \begin{align*}
                \frac{ \EE \left( \RPI_n^{(2)}[x] (\bm{v}) \right) }{ \EE \left(\max \left\{ v_{\sigma(i)} : i \in [1 \until n] \right\} \right) } 
                    &= \sum_{t = \lceil xn \rceil + 1}^{n - 1} 
                        \frac{ \EE \left( \RPI_n^{(2)}[x] (\bm{v}) \mid T[x] = t \right) }{ \EE \left(\max \left\{ v_{\sigma(i)} : i \in [1 \until n] \right\} \right) }
                        \frac{\lceil xn \rceil}{t (t - 1)} \\
                        &\quad\quad + 
                        \frac{ \EE \left( \RPI_n^{(2)}[x] (\bm{v}) \mid T[x] = n \right) }{ \EE \left(\max \left\{ v_{\sigma(i)} : i \in [1 \until n] \right\} \right) }
                        \left( \frac{\lceil xn \rceil}{n (n - 1)} + \frac{\lceil xn \rceil}{n} \right) \\
                    &\xrightarrow[n \to \infty]{} \int_{ x }^{ 1 } \alpha \left( p = \frac{u}{1 + u} \right) \frac{x}{u^2} du + \alpha \left( p = \frac{1}{2} \right) x \,.
            \end{align*}
        
            Finally, by \Cref{Result: Equivalent form of limit performance}, we conclude that
            \begin{align*}
                \perf(\sRPI[x])
                    &= \liminf_{n \to \infty} \inf_{\bm{v}}
                        \frac{ \EE(\RPI_n^{(1)}({\bm{v}})) }{ 2 \EE \left(\max \left\{ v_{\sigma(i)} : i \in [1 \until n] \right\} \right) } 
                        + \frac{ \EE(\RPI_n^{(2)}({\bm{v}})) }{ 2 \EE \left(\max \left\{ v_{\sigma(i)} : i \in [1 \until n] \right\} \right) } \\
                    &\ge \liminf_{n \to \infty} \inf_{\bm{v}} - \frac{x}{2} \ln(x) + \frac{ \EE(\RPI_n^{(2)}({\bm{v}})) }{ 2 \EE \left(\max \left\{ v_{\sigma(i)} : i \in [1 \until n] \right\} \right) } \\
                    &= \frac{x}{2} \left( -\ln(x) + \int_{x}^{1} \alpha \left( \frac{u}{1 + u} \right) \frac{1}{u^2} du + \alpha \left( \frac{1}{2} \right) \right) \,,
            \end{align*}
            which proves the result.
        \end{proof}
    }
    
    Finally, optimizing over $x$, we derive the lower bound that our algorithm implies.
    
    \begin{proof}[Proof of \Cref{Result: Algorithm lowerbound}]
        By \Cref{Result: Limit performance of RPI[x]}, all we need to show is that there exists $x^* \in [0, 1]$ such that $\perf(\sRPI[x^*]) \ge 0.495$. Consider a lower bound $\underline{\alpha}$ for $\alpha$. Then,
        \[
            \perf(\sRPI[x]) \ge \frac{x}{2} \left( -\ln(x) + \int_{x}^{1} \underline{\alpha} \left( \frac{u}{1 + u} \right) \frac{1}{u^2} du + \underline{\alpha} \left( \frac{1}{2} \right) \right) \,.
        \]
        In particular, since $\alpha$ is an increasing function, we can take the following expression for $\underline{\alpha}$.
        \[
            \underline{\alpha}(p) = \sum_{k = 1}^K \alpha(p_{k - 1}) \1[p \in (p_{k - 1}, p_k]] \,,
        \]
        where $p_0 = 0 < p_1 < p_2 < \ldots < p_{K} = 1$. The values $(\alpha(p_{k - 1}))_{k \in [1 \until K]}$ can be approximated from below by following the procedure of~\cite[Appendix B.10]{correa2020SampledrivenOptimalStopping}. To compute a lower bound for $\alpha(p)$, we only need to solve a linear program. The more variables has this linear program, the better the approximation of $\alpha(p)$. Therefore, we can have a procedure to give an arbitrary approximation of $\perf(\sRPI[x])$ by increasing the number $K$ and the number of variables in the related linear program for each $p_k \in [0, 1]$.
    
        In particular, taking $K = 4000$ and $p_k = k / 4000$, we can consider the bounds given in~\cite[Appendix B.10]{correa2020SampledrivenOptimalStopping} for the corresponding $\alpha(p_k)$. With these values, for $x^* \approx 0.457$ we get that $\perf(\sRPI[0.441]) \ge 0.495$, which proves the result.
    \end{proof}
}

\iftoggle{SODA}
{
    \vspace{-1em}
    \section{Upper Bound}
    \label{Section: Upper bound}
    \vspace{-0.5em}
}{
    \section{Upper Bound}
    \label{Section: Upper bound}
}

In this section, we present an upper bound for the limit performance of any algorithm for the value-observed variant of the Repeated Prophet Inequality. \Cref{Result: Algorithms upperbound} states that all algorithms have a limit performance of at most $0.502$. 

\smallskip
\noindent \emph{Overview of the proof.} 
The proof is based on analyzing the Repeated Secretary Problem. Considering this new goal for the DM, we characterize the optimal algorithm and compute its performance. Our characterization of the optimal algorithm follows the spirit of~\cite{gilbert1966RecognizingMaximumSequence}. To conclude, \Cref{Result: Choosing the maximum upperbound} shows that the limit performance of any algorithm for the Repeated Prophet Inequality is bounded by its probability of accepting the item with maximum value. Therefore, the performance of the optimal algorithm for the Repeated Secretary Problem is an upper bound for the Repeated Prophet Inequality.

Before we prove \Cref{Result: Algorithms upperbound}, let us argue why it does not follow immediately from classical results.

\smallskip
\noindent \emph{Classical $k$-choices Secretary Problem.} Consider the following optimal stopping problem. An adversary chooses a list of $n$ items with associated real values $v_1 \ge v_2 \ge \ldots \ge v_{n} \ge 0$ and nature chooses a random permutation $\sigma \colon [1 \until n] \to [1 \until n]$. Then, items are presented to the DM one by one according to $\sigma$, i.e., $v_{\sigma(1)}, \allowbreak v_{\sigma(2)}, \allowbreak \ldots, \allowbreak v_{\sigma(n)}$. When the $i$-th item arrives, the DM does not have access to the numerical value $v_{\sigma(i)}$ but can compare it with the previous items, i.e. $v_{\sigma(1)}, v_{\sigma(2)}, \ldots, v_{\sigma(i - 1)}$. In other words, she knows its relative ranking within the items observed so far. After the arrival of the $i$-th item, the DM must make an irrevocable decision to either accept the item or reject the item forever and proceed to the next item, if any. When accepting an item, its value is not revealed to the DM.
Since items may have equal values, we assume that there is an arbitrary tie-breaking rule that is consistent with the relative ranks revealed and publicly selected before the process starts.
Her goal is to maximize the probability of choosing the item with the largest value while accepting only $k$ items. 

\smallskip
\noindent \emph{Characterization of~\cite{gilbert1966RecognizingMaximumSequence}.} 
The optimal algorithm for the classical $k$-choice Secretary problem is described in~\cite{gilbert1966RecognizingMaximumSequence}. For the $1$-choice variant of the problem, i.e. when the DM can only accept one item, the optimal strategy is to observe a fixed portion of the items (asymptotically $1/e$) and then accept a record if it appears. For this strategy, the probability of accepting $v_1$ converges to $1/e \approx 0.368$ as the number of items goes to infinity. For the $2$-choices variant of the problem, i.e. when the DM can only accept two items, the optimal strategy proceeds as follows. First, it observes a fixed portion of the items (asymptotically $1 / e^{3/2}$) and then accepts a record if it appears. Second, after accepting an item for the first time, if the portion of revealed items has not reached another fixed portion (asymptotically $1/e$), then items are observed but not accepted until this portion of items has been revealed. Finally, it accepts a record if it appears. For this strategy the probability of accepting $v_1$ converges to $1 / e^{3/2} + 1/e \approx 0.591$ as the number of items goes to infinity.

\smallskip
\noindent \emph{Issues of~\cite{gilbert1966RecognizingMaximumSequence} in our setting.} There is a fundamental difference between the classical $2$-choices Secretary Problem and the Repeated Secretary Problem. In the classical problem, all items are revealed to the DM independently of her strategy. In contrast, in our setting, accepting an item in Phase~1 constrains the items that will be revealed to the DM. In particular, after accepting an item during Phase~1, the rest of the items in Phase~1 are not revealed to the DM. But there is a more important issue with applying the conclusions of~\cite{gilbert1966RecognizingMaximumSequence} in our setting.
Even if all items were revealed in the Repeated Secretary Problem, the DM can accept an item for the second time only during Phase~2. In particular, the DM could not accept a second item before observing half of the total number of items. Recall that the optimal strategy of~\cite{gilbert1966RecognizingMaximumSequence} indicates that a second record should be chosen as soon as a portion of $1/e$ of the items has been revealed. Therefore, the optimal strategy of~\cite{gilbert1966RecognizingMaximumSequence} can not be implemented in the Repeated Secretary Problem.

\iftoggle{SODA}
{

}{
    \smallskip
    \noindent \emph{Modern sample-driven Secretary Problem.} In modern sample-driven Secretary Problem, an adversary chooses a sequence of $n$ items with values $v_1, \ldots, v_n$. Then, the process follows in two phases. During the first phase (sampling phase), some items are revealed to the DM, but she can never accept them. During the second phase (online phase), the other items are revealed to the DM in a random order and online fashion. For each item, the DM must make an irrevocable decision to either accept the item and stop the process or reject the item forever and proceed to the next item. The goal of the DM is to accept $v_1$. The work of~\cite{kaplan2019CompetitiveAnalysisSample} presents an upper bound for all algorithms when the number of items revealed during the sampling phase is fixed. The work of~\cite{correa2021SecretaryProblemIndependent} presents an optimal algorithm when the items revealed during the sampling phase are randomly chosen according to a fixed probability. 
    
    \smallskip
    \noindent \emph{Issues of modern sample-driven Secretary Problem.} One attempt to analyze Phase~2 of the Repeated Secretary Problem is to model it as a sample-driven Secretary Problem, considering the items revealed during Phase~1 as samples. But this approach is doomed to fail since the goal of the DM is different from the sample-driven Secretary Problem. Indeed, in the repeated Secretary Problem, the goal of the DM is to accept the item with the largest value. In particular, during Phase~2, the DM should only accept items whose value is larger than any item revealed so far, even when compared with items from Phase~1. In contrast, in the sample-driven Secretary Problem, the goal of the DM is to accept the item with the largest value only among the items in the online phase. That is why this approach can not work. As a concrete example, consider the Repeated Secretary Problem and assume that: (1)~the DM has not accepted any item during Phase~2 and the last item $v_{\sigma(2n)}$ is revealed to her; (2)~the last item $v_{\sigma(2n)}$ has the largest value among all the items in Phase~2; and (3)~an item revealed during Phase~1 has a value larger than $v_{\sigma(2n)}$. Following the interpretation of Phase~2 as a sample-driven Secretary Problem, the DM achieves her goal by accepting the last item revealed, as it is the item with the largest value in Phase~2 (interpreted as the online phase). But this is not true for the Repeated Secretary Problem: $v_{\sigma(2n)}$ is not the largest possible value and therefore the DM has no incentives to accept this item. In conclusion, it is incorrect to interpret Phase~2 of the Repeated Secretary Problem as a sample-driven Secretary Problem.
}

In the rest of this section, we consider the following (value-oblivious) algorithms.

\smallskip
\noindent \emph{Waiting Algorithm.} For $x \in [0, 1]$, denote $\sWAI[x]$ the algorithm that proceeds for an instance of size $2n$ as follows. During Phase~1, $\WAI[x]_{n}$ observes but does not accept the first $\lceil xn \rceil$ items and then accepts any record that appears. After observing during Phase~1 a total of $t$ items, during Phase~2, $\WAI[x]_{n}$ observes but does not accept items until $\lceil (n + t) / e \rceil$ items have been observed in total. Then, it accepts any record that appears.

In \Cref{Section: Characterization of the optimal algorithm}, we characterize the optimal algorithm when the DM maximizes the probability of choosing the item with the largest value as a Waiting algorithm. In \Cref{Section: Computation of the upper bound}, for any Waiting algorithm, we compute its probability of choosing the items with the largest value and optimize to derive the numerical upper bound of $0.502$, proving \Cref{Result: Algorithms upperbound}.

\subsection{Characterization of the optimal algorithm}
\label{Section: Characterization of the optimal algorithm}

In this section, we characterize the optimal algorithm for the Repeated Secretary Problem, where the goal of the DM is to maximize the probability of choosing the item with the largest value. The optimal algorithm proceeds as follows. In Phase~1, a fixed portion of the items is observed but not accepted and then a record is accepted if it appears. In Phase~2, depending on how many items were revealed to the DM in Phase~1, more items may be observed but not accepted and then a record is accepted if it appears. Formally, we have the following result.

\begin{lemma}[Characterization of the optimal algorithm]
\label{Result: Characterization of the optimal algorithm}
    There exists $x^* \in [0, 1]$ such that $\sWAI[x^*]$ is optimal for the Repeated Secretary Problem. Formally,
    \[
        \sup_{\sALG} \liminf_{n \to \infty} \PP( \ALG_{n} \text{ accepts } v_1 ) 
            = \max_{x \in [0, 1]} \liminf_{n \to \infty} \PP( \WAI[x]_{n} \text{ accepts } v_1 ) \,.
    \]
\end{lemma}

\iftoggle{SODA}
{
    \begin{proof}[Sketch proof]
        Following the spirit of the proof of \Cref{Result: Acceptance probability of Secretary algorithms} in~\cite{gilbert1966RecognizingMaximumSequence}, we characterize the optimal algorithm by a series of observations that the algorithm needs to satisfy. First, we show that it must accept only records. Then, during Phase~1, if a record would be chosen at position $t$, then a record would be chosen at any later position. Similarly, during Phase~2, after observing the same number of items during Phase~1, if a record would be chosen at position $j$, then a record would be chosen at any later position. Lastly, we deduce the optimal observation period for Phase~2 based on \Cref{Result: Optimal strategy in the Secretary Problem}.
    \end{proof}
}{
    \begin{proof}
        We characterize the optimal algorithm by a series of observations. Recall that $v_1$ is the item with the largest value, the only item the DM is interested in.
        \begin{enumerate}
            \item \emph{Choose only records.} The optimal algorithm accepts only records, i.e. items with the largest value among the items revealed so far. Indeed, items which are not records are necessarily different from $v_1$. Moreover, there is always a positive probability that the last item revealed is $v_1$, formally, $\PP(\sigma^{-1}(1) = 2n) = (2n)^{-1} > 0$. Therefore, the optimal algorithm accepts only records.
    
            \item \emph{Choose late records in Phase~1.} During Phase~1, if the optimal algorithm accepts a record that appears in position $t$, then it accepts a record that appears in any later position. Indeed, the probability that a record at position $t$ is $v_1$ is equal to $t / 2n$, an increasing function of $t$. Moreover, the probability to accept $v_1$ in Phase~2 when $t$ items have been revealed during Phase~1 is also an increasing function of $t$. Therefore, if the optimal algorithm accepts a record at position $t$ during Phase~1, it also accepts a record at any later position. 
    
            \item \emph{Choose late records in Phase~2.} Fix the number of items revealed during Phase~1. Then, during Phase~2, if the optimal algorithm accepts a record that appears in position $j$, the optimal algorithm accepts a record that appears in any later position. Indeed, assume that the number of items revealed during Phase~1 is $t$. Then, the probability that a record at position $j$ is $v_1$ is equal to $(t + j) / 2n$, an increasing function of $j$. Therefore, after fixing the number of revealed items during Phase~1, if the optimal algorithm accepts a record at position $j$ during Phase~2, then it also accepts a record at any later position.
    
            \item \emph{Optimal observation in Phase~2.} Fix the number of items revealed during Phase~1 to $t \in [1 \until n]$. Then, during Phase~2, the optimal algorithm accepts a record that appears in position $j$ if and only if $(t + j) > \lceil (n + t) / e \rceil$. Indeed, there are $(n + t)$ items in total and the first $t$ items can not be accepted by the DM (since they belong to Phase~1). Therefore, by \Cref{Result: Optimal strategy in the Secretary Problem}, using $\rho = t / (n + t)$, a asymptotically optimal strategy accepts any record after $\max( t, \lceil (n + t) / e \rceil )$ items has been revealed. In other words, during Phase~2, a record that appears in position $j$ is accepted by the optimal algorithm if and only if $(t + j) > \lceil (n + t) / e \rceil$.
        \end{enumerate}
    
        Putting all previous observations together, we have characterized the optimal algorithm up to the first portion of observed items $x \in [0, 1]$. Formally, define the function $f\colon [0, 1] \to [0, 1]$ such that, for all $x \in [0, 1]$, 
        \[
            f(x) \defas \liminf_{n \to \infty} \PP( \WAI[x]_{n} \text{ accepts } v_1 ) \,.
        \]
        Then, we have that
        \[
            \sup_{\sALG} \liminf_{n \to \infty} \PP( \ALG_{n} \text{ accepts } v_1 ) 
                = \sup_{x \in [0, 1]} f(x) \,.
        \]
        
        To finish the proof, note that $f$ is continuous. Therefore, there exists a point in which it achieves its supremum, i.e. there exists $x^* \in [0, 1]$ such that $\sup_{x \in [0, 1]} f(x) = f(x^*)$. 
    \end{proof}
}

\begin{Remark}
    By \Cref{Result: Characterization of the optimal algorithm}, the optimal algorithm for the Repeated Secretary Problem is value-oblivious. Therefore, unsurprisingly, the value-oblivious and value-observed variants of the Repeated Secretary Problem are equivalent.
\end{Remark}

\subsection{Computation of the upper bound}
\label{Section: Computation of the upper bound}

In this section, we compute the optimal limit probability of accepting the item with the largest value in the Repeated Secretary Problem. Then, we prove \Cref{Result: Algorithms upperbound} which states that the limit performance of any algorithm for the Repeated Prophet Inequality is at most $0.502$. Based on the characterization of the optimal algorithm given by \Cref{Result: Characterization of the optimal algorithm}, we only need to look at a uniparametric family of algorithms. We derive the exact limit probability for each one of these algorithms. Then, using standard optimization techniques, we deduce that the limit probability of the optimal algorithm for the Repeated Secretary problem is at most $0.502$. Finally, by \Cref{Result: Choosing the maximum upperbound}, we conclude that $0.502$ is also an upper bound for the Repeated Prophet Inequality.

We start by computing the limit probability for each algorithm in the uniparametric family.
\begin{lemma}[Limit probability of waiting algorithms]
\label{Result: Limit probability of algorithms}
    For all $x \in [0, 1]$, we have that
    \begin{align*}
        \lim_{n \to \infty} &\PP( \WAI[x]_{n} \text{ accepts } v_1 ) \\ 
            &= - \frac{x}{2} \ln(x) + \frac{x}{2} \left( \int_{ \min(x, (e - 1)^{-1}) }^{(e - 1)^{-1}} \frac{ u + 1 }{ e u^2 } du 
                + \int_{ \max(x, (e - 1)^{-1}) }^{1} \frac{1}{u} \ln \left( 1 + \frac{1}{u} \right) du \right) + \frac{x}{2} \ln(2) \,.
    \end{align*}
\end{lemma}

\iftoggle{SODA}
{
    \begin{proof}[Sketch proof]
        We partition on the cases where $\sigma$ assigns $v_1$ to Phase~1 or Phase~2. Then, the algorithm $\WAI[x]_n$ acts similarly to a Secretary strategy during Phase~1. Therefore, we can deduce the first term $- x \ln(x) / 2$. For Phase~2, the analysis is more involved. By conditioning on the number of items revealed during Phase~1, we can reason about the behaviour of $\WAI[x]_n$ during Phase~2. In particular, if the DM observed $t \in [1 \until n]$ items, then the behaviour of $\WAI[x]_n$ is similar to some Secretary strategy. The exact Secretary strategy depends on the value of $t$. Replacing the correct expression for each case, summing up over all values of $t$ and taking the limit as $n$ goes to $\infty$ we deduce the second term involving integral. Note that the last term $x \ln(2) /2$ corresponds to the case in which $\sigma$ assigned $v_1$ to Phase~2, the DM did not accept any item during Phase~1 and, applying the corresponding Secretary strategy with $n$ samples, the DM accepted $v_1$. 
    \end{proof}
}{
    \begin{proof}
        Fix an arbitrary instance $\bm{v} = (v_1, v_2, \ldots, v_{2n})$ of size $2n$. Then, by partitioning on the events that $\sigma$ assigns $v_1$ to Phase 1 or Phase 2, we have that
        \begin{align*}
        	\PP(\WAI[x]_{n} &\text{ accepts } v_1 ) \\
                &= \PP(\WAI[x]_{n} \text{ accepts } v_1, \sigma^{-1}(1) \le n) 
                    + \PP(\WAI[x]_{n} \text{ accepts } v_1, \sigma^{-1}(1) > n)  \,.
        \end{align*}
    	
        By definition of $\WAI[x]_n$, during Phase~1, the algorithm observes the first $\lceil xn \rceil$ items and then accepts the first record that appears. Conditioning on $\sigma$ assigning $v_1$ to Phase~1, i.e. $\sigma^{-1}(1) \le n$, we have that $\WAI[x]_n$ accepts $v_1$ with the same probability that certain Secretary algorithm would accept an item with maximum value. To be more precise, consider a new instance $\bm{v}' = (v_1', v_2', \ldots, v_{n}')$ of size $n$. Then, the strategy $\SEC[x]_{n}$ accepts $v_1'$ with the same probability that $\WAI[x]_n$ accepts $v_1$, conditined on $v_1$ being assigned to Phase~1. Formally,
        \[
            \PP(\WAI[x]_{n} \text{ accepts } v_1 \mid \sigma^{-1}(1) \le n) = \PP(\SEC[x]_{n} \text{ accepts } v_1' ) \,.
        \]
        Therefore, by \Cref{Result: Acceptance probability of Secretary algorithms}, we have that
        \begin{align*}
            \PP(\WAI[x]_{n} \text{ accepts } v_1, \sigma^{-1}(1) \le n)
                = \frac{1}{2} \PP(\WAI[x]_{n} \text{ accepts } v_1 \mid \sigma^{-1}(1) \le n)
                \xrightarrow[n \to \infty]{} - \frac{x}{2} \ln(x) \,.
        \end{align*}
    
        The computation of the limit of $\PP(\WAI[x]_{n} \text{ accepts } v_1, \sigma^{-1}(1) > n)$ is more involved and proceeds as follows. After fixing the number of items revealed during Phase~1 and the position in which $v_1$ might appear, we can use the definition of $\WAI[x]_{n}$ in Phase 2 to interpret its behaviour as an optimal Secretary strategy in a particular instance. Then, by \Cref{Result: Acceptance probability of Secretary algorithms}, we know the asymptotic expression of the corresponding (conditional) probability of accepting $v_1$. Finally, summing up all possibilities and taking the limit as $n$ grows, we will conclude the desired expression.
    
        Denote the (random) total number of items revealed during Phase~1 by $T[x]$. Then, 
        \begin{align}
        \label{Eq: Total probability}
            \begin{aligned}
                \PP(\WAI[x]_{n} &\text{ accepts } v_1, \sigma^{-1}(1) > n) \\
                    &= \PP(\WAI[x]_{n} \text{ accepts } v_1 \text{ during Phase 2}) \\
                    &= \sum_{t = \lceil xn \rceil + 1}^{n} \PP(\WAI[x]_{n} \text{ accepts } v_1 \text{ during Phase 2}, T[x] = t) \,.
            \end{aligned}
        \end{align}

        Fix $t \in [ \lceil xn \rceil + 1 \until n]$. Then, for $\WAI[x]_{n}$ to accept $v_1$ during Phase~2, it is needed that $\sigma$ assigns a position to $v_1$ during Phase 2. In particular, 
        \begin{align}
        \label{Eq: accepting while seeing i}
            \begin{aligned}
                \PP(\WAI[x]_{n} &\text{ accepts } v_1 \text{ during Phase 2}, T[x] = t) \\
                    &= \PP(\WAI[x]_{n} \text{ accepts } v_1 \text{ during Phase 2}, T[x] = t, \sigma^{-1}(1) \in [(n + 1) \until 2n]) \\
                    &= \PP(\WAI[x]_{n} \text{ accepts } v_1 \text{ during Phase 2}, T[x] = t, \sigma^{-1}(1) \in [1 \until i] \cup [(n + 1) \until 2n]) \,, 
        	\end{aligned}
        \end{align}    
        where the last equality only adds an event of probability zero.
    
        Recall that $\WAI[x]_{n}$, after observing during Phase~1 a total of $t$ items, during Phase~2 $\WAI[x]_{n}$ observes but does not accept items until $\lceil (n + t) / e \rceil$ items have been observed in total and then accepts any record that appears. In particular, it has the same behaviour as a Secretary strategy that, for an instance of size $(n + t)$, observes but does not accept the first $\max \{ t,  \lceil (n + t) / e \rceil\}$ items and then accepts any record that appears. Formally, for every $n$ and $t \in [1 \until n]$, consider a new instance $\bm{v}' = ( v_1', v_2', \ldots, v_{n + t}')$ of size $(n + t)$. Then, denote the event that the corresponding Secretary strategy accepts the item with the largest value in an instance of size $(n + t)$ by $E[t, n + t]$. Formally,
        \[
            E[t, n + t] \defas \SEC \left[ \frac{\max \{ t,  \lceil (n + t) / e \rceil\} }{ n + t } \right]_{(n + t)} \text{ accepts } v_1' \,.
        \]
        Finally, by the previous argument, we have the following equality.
        \begin{align}
        \label{Eq: conditional probability}
            \begin{aligned}
                \PP(\WAI[x]_{n} &\text{ accepts } v_1 \text{ during Phase 2} \mid T[x] = t, \sigma^{-1}(1) \in [1 \until i] \cup [(n + 1) \until 2n]) \\
                    &= \PP(E[t, n + t]) \,.
            \end{aligned}            
        \end{align}
        
        Note that, since $\WAI[x]_{n}$ is value-oblivious, the (random) number of items observed during Phase~1, denoted $T[x]$, is independent of the position of $v_1$. Formally, 
        \begin{align}
        \label{Eq: probability}
            \begin{aligned}
                \PP( T[x] = t, \sigma^{-1}(1) \in [1 \until t] \cup [(n + 1) \until 2n] ) 
                    &= \PP( T[x] = t ) \cdot \PP( \sigma^{-1}(1) \in [1 \until t] \cup [(n + 1) \until 2n] ) \\
                    &= \left( \frac{1}{t} \frac{\lceil xn \rceil}{t - 1} \1[t > \lceil xn \rceil] + \frac{\lceil xn \rceil}{n} \1[t = n]\right) 
                        \cdot \left( \frac{ n + t }{ 2n } \right) \,,
            \end{aligned}
        \end{align}
        where the last equality follows from \Cref{Result: Distribution of stopping time}.
        
        Therefore, applying~\eqref{Eq: conditional probability}~and~\eqref{Eq: probability} to \eqref{Eq: accepting while seeing i}, we have that
        \begin{align}
        \label{Eq: accepting while seeing i, replaced}
            \begin{aligned}
                \PP(\WAI[x]_{n} &\text{ accepts } v_1 \text{ during Phase 2}, T[x] = t) \\
                    &= \PP(E[t, n + t]) \cdot \left( \frac{1}{t} \frac{\lceil xn \rceil}{t - 1} \1[t > \lceil xn \rceil] + \frac{\lceil xn \rceil}{n} \1[t = n]\right) 
                        \cdot \left( \frac{ n + t }{ 2n } \right) \,.
            \end{aligned}
        \end{align}
    
        Note that, fixing $t/n = u \in [0, 1]$, by \Cref{Result: Acceptance probability of Secretary algorithms}, we have that 
        \begin{align}
        \label{Eq: limit acceptance of secretary}
            \begin{aligned}
                \lim_{n \to \infty} \PP(E[t, n + t])  
                    &= -\left( \frac{u}{1 + u} \right) \ln \left( \frac{u}{1 + u} \right) \1[ue  > (1 + u)]
                        -\left( \frac{1}{e} \right) \ln \left( \frac{1}{e} \right) \1[ue  \le (1 + u)] \\
                    &= -\left( \frac{u}{1 + u} \right) \ln \left( \frac{u}{1 + u} \right) \1[ue  > (1 + u)]
                        + \left( \frac{1}{e} \right) \1[ue  \le (1 + u)] \,.
            \end{aligned}
        \end{align}
    
        Finally, replacing~\eqref{Eq: accepting while seeing i, replaced} back to~\eqref{Eq: Total probability}, we obtain that
        \begin{align*}
            \PP(\WAI[x]_{n} &\text{ accepts } v_1, \sigma^{-1}(1) > n) \\
                &= \sum_{t = \lceil xn \rceil + 1}^n \PP(E[t, n + t]) 
                    \left( \frac{1}{t} \frac{\lceil xn \rceil}{t - 1} \1[i > \lceil xn \rceil] + \frac{\lceil xn \rceil}{n} \1[t = n]\right) 
                    \left( \frac{ n + t }{ 2n } \right) \\
                &= \sum_{t = \lceil xn \rceil + 1}^{n-1} \PP(E[t, n + t]) 
                    \left( \frac{1}{t} \frac{\lceil xn \rceil}{t - 1} \right) 
                    \left( \frac{ n + t }{ 2n } \right) \\
                &\quad\quad+ \PP(E[n, 2n])  \frac{\lceil xn \rceil}{n} \\
                &= \sum_{t = \lceil xn \rceil + 1}^{n-1} \PP(E[t, n + t]) 
                    \left( \frac{1}{t} \frac{\lceil xn \rceil}{t - 1} \right) 
                    \left( \frac{ n + t }{ 2n } \right) \\
                &\quad\quad+ \PP(E[n, 2n])  \frac{\lceil xn \rceil}{n} \,.
        \end{align*}
        Therefore, taking the limit as $n$ grows and using~\eqref{Eq: limit acceptance of secretary}, we have that
        \begin{align*}
            \lim_{n \to \infty} &\PP(\WAI[x]_{n} \text{ accepts } v_1, \sigma^{-1}(1) > n) \\
                &= \int_{x}^{1} 
                    \left( -\left( \frac{u}{1 + u} \right) \ln \left( \frac{u}{1 + u} \right) \1[ue  > (1 + u)]
                        + \left( \frac{1}{e} \right) \1[ue  \le (1 + u)] \right)
                    \left( \frac{x}{u^2} \right) 
                    \left( \frac{ 1 + u }{ 2 } \right) du \\
                &\quad\quad- \frac{1}{2} \ln \left( \frac{1}{2} \right) x \\
                &= \int_{x}^{1} 
                    \left( \frac{x}{2u} \right) \ln \left( \frac{1 + u}{u} \right) \1[ue  > (1 + u)]
                        + \left( \frac{x (1 + u)}{2 e u^2} \right) \1[ue  \le (1 + u)] du + \frac{x}{2} \ln(2) \\
                &= \frac{x}{2} \int_{ \min\{ x, (1 - e)^{-1} \} }^{ (1 - e)^{-1} } \frac{u + 1}{e u^2} du
                    + \frac{x}{2} \int_{ \max\{ x, (1 - e)^{-1} \} }^{ 1 } 
                    \frac{ 1 }{ u } \ln \left( 1 + \frac{1}{u} \right) du + \frac{x}{2}  \ln(2) \,.
        \end{align*}
        
        We conclude that, for all $x \in [0, 1]$, we have
        \begin{align*}
        	\lim_{n \to \infty} &\PP(\WAI[x]_{n} \text{ accepts } v_1 ) \\
                &= \lim_{n \to \infty} \PP(\WAI[x]_{n} \text{ accepts } v_1, \sigma^{-1}(1) \le n) 
                    + \PP(\WAI[x]_{n} \text{ accepts } v_1, \sigma^{-1}(1) > n) \\ 
                &= - \frac{x}{2} \ln(x) + \frac{x}{2} \left( \int_{ \min(x, (e - 1)^{-1}) }^{(e - 1)^{-1}} \frac{ u + 1 }{ e u^2 } du 
                    + \int_{ \max(x, (e - 1)^{-1}) }^{1} \frac{1}{u} \ln \left( 1 + \frac{1}{u} \right) du \right) + \frac{x}{2} \ln(2) \,.
        \end{align*}
    \end{proof}
}

We proceed to prove \Cref{Result: Algorithms upperbound} which states that the limit performance of any algorithm for the Repeated Prophet Inequality is at most $0.502$. 

\iftoggle{SODA}
{
    \begin{proof}[Sketch proof of \Cref{Result: Algorithms upperbound}]
        Applying \Cref{Result: Choosing the maximum upperbound}, \Cref{Result: Characterization of the optimal algorithm} and \Cref{Result: Limit probability of algorithms}, the performance of any algorithm can be upper bounded by $\max_{x \in [0, 1]} f(x)$, where $f(x)$ is the limit acceptance probability of $v_1$ for $\sWAI[x]$, given in \Cref{Result: Limit probability of algorithms}. Since $f$ is striclty concave, standard optimization techniques can be applied to deduce the numerical bound of $0.502$ at $x^* \approx 0.463$.
    \end{proof}
}{
    \begin{proof}[Proof of \Cref{Result: Algorithms upperbound}]
        Consider any algorithm $\sALG$ for the Repeated Secretary Problem. By \Cref{Result: Choosing the maximum upperbound}, we have that
        \[
            \perf(\sALG) \le \liminf_{n \to \infty} \PP( \ALG_n \text{ accepts } v_1 ) \,.
        \]
        Then, applying \Cref{Result: Characterization of the optimal algorithm}, we moreover have that
        \[
            \perf(\sALG) \le \max_{x \in [0, 1]} \liminf_{n \to \infty} \PP( \WAI[x]_n \text{ accepts } v_1 ) \,.
        \]
        Finally, by \Cref{Result: Limit probability of algorithms}, we have that
        \begin{align*}
            &\perf(\sALG) \\
                &\le \sup_{x \in [0, 1]} 
                    - \frac{x}{2} \ln(x) + \frac{x}{2} \left( \int_{ \min(x, (e - 1)^{-1}) }^{(e - 1)^{-1}} \frac{ u + 1 }{ e u^2 } du 
                    + \int_{ \max(x, (e - 1)^{-1}) }^{1} \frac{1}{u} \ln \left( 1 + \frac{1}{u} \right) du \right) + \frac{x}{2} \ln(2) \\
                &\safed \sup_{x \in [0, 1]} f(x) \,.
        \end{align*}
        Therefore, all we need to show is that
        \[
            \sup_{x \in [0, 1]} f(x) \le 0.502 \,.
        \]
        We prove that $f$ is strictly concave. Therefore, standard optimization techniques apply~\cite{bertsekas2015ConvexOptimizationAlgorithms}.
    
        Define the function $g \colon [0, 1] \to \RR$ by
        \[
            g(x) \defas x\left( \int_{ \min(x, (e - 1)^{-1}) }^{(e - 1)^{-1}} \frac{ u + 1 }{ e u^2 } du 
                    + \int_{ \max(x, (e - 1)^{-1}) }^{1} \frac{1}{u} \ln \left( 1 + \frac{1}{u} \right) du \right) \,.
        \]
        Then, $f(x) = (- x \ln(x) + g(x) + x \ln(2)) / 2$, i.e. the sum of a strictly concave function and $g$. Therefore, if $g$ is concave, then $f$ is strictly concave. So all we need to do is to show that $g$ is concave. Note that
        \[
            g(x) = x \left( \int_{ x }^{ 1 } \frac{ u + 1 }{ e u^2 } \1[ue  \le (1 + u)] + \frac{1}{u} \ln \left( 1 + \frac{1}{u} \right) \1[ue  > (1 + u)] du \right)
        \]
        In particular, $g$ is the integral of a continuous function. Therefore, $g$ is differentiable and its derivative $g'$ is a continuous function. We show that $g$ is concave by proving that $g'$ is decreasing.
    
        For $x \in (0, (e - 1)^{-1})$, we have that
        \begin{align*}
            g(x) &= x \int_{ x }^{ (e - 1)^{-1} } \frac{ u + 1 }{ e u^2 } du + x \int_{(e - 1)^{-1}}^{1} \frac{1}{u} \ln \left( 1 + \frac{1}{u} \right) du \,, \\
            g'(x) &= \int_{ x }^{ (e - 1)^{-1} } \frac{ u + 1 }{ e u^2 } du - \frac{ x + 1 }{ e x } + \int_{(e - 1)^{-1}}^{1} \frac{1}{u} \ln \left( 1 + \frac{1}{u} \right) du \,, \\
            g''(x) &= - \frac{ x + 1 }{ e x^2 } - \frac{x - (x + 1)}{ex^2} = - \frac{1}{ex} < 0 \,.
        \end{align*}    
        Therefore, $g'$ is decreasing in $(0, (e - 1)^{-1})$.
    
        For $x \in ((e - 1)^{-1}, 1)$, we have that
        \begin{align*}
            g(x) &= x \int_{x}^{1} \frac{1}{u} \ln \left( 1 + \frac{1}{u} \right) du \,, \\
            g'(x) &= \int_{x}^{1} \frac{1}{u} \ln \left( 1 + \frac{1}{u} \right) du - \ln \left( 1 + \frac{1}{x} \right) \,, \\
            g''(x) &= - \frac{1}{x} \ln \left( 1 + \frac{1}{x} \right) + \frac{1}{x (x + 1) } = \frac{1}{x} \left( \frac{1}{x + 1} - \ln \left( 1 + \frac{1}{x} \right) \right) < 0 \,.
        \end{align*}
        Therefore, $g'$ is decreasing in $((e - 1)^{-1}, 1)$.
        
        Since $g'$ is continuos, we deduce that $g'$ is strictly decreasing in $[0, 1]$. Therefore $g$ is concave and so is $f$. Finally, we can apply any standard optimization algorithm to deduce the bound of $0.502$ at $x^* \approx 0.463$. For example, we can use binary search to $f'$ over $[0, 1]$ to find the unique maximum, as it readily gives error guarantees.
    \end{proof}
}

\begin{Remark}
    As mentioned before, we obtained an optimal algorithm for the Repeated Secretary Problem. Indeed, we proved that there is $x^* \in [0, 1]$ such that $\sWAI[x^*]$ is optimal and showed that $x^* \approx 0.463$, the maximizer of the function in \Cref{Result: Limit probability of algorithms}
\end{Remark}

\iftoggle{SODA}
{
    \vspace{-1em}
    \section{Convex Combination of Phases}
    \label{Section: Convex Combination of Phases}
    \vspace{-0.5em}
}{
    \section{Convex Combination of Phases}
    \label{Section: Convex Combination of Phases}
}
    
In this section, we generalize our results to the case where the goal of the DM is not simply the sum of the values chosen in each phase, but a convex combination of the ratio obtained in each phase. Previously, we have only considered a very particular convex combination, the average. More generally, define for each $\lambda \in [0, 1]$, the limit $\lambda$-performance by 
\[
    \perf_\lambda(\sALG) 
        \defas \liminf_{n \to \infty} \inf_{\bm{v}} 
            (1 - \lambda) \frac{ \EE(\ALG_n^{(1)}({\bm{v}})) }{ \EE \left(\max \left\{ v_{\sigma(i)} : i \in [1 \until n] \right\} \right) } + \lambda \frac{ \EE(\ALG_n^{(2)}({\bm{v}})) }{ \EE \left(\max \left\{ v_{\sigma(i)} : i \in [(n + 1) \until 2n] \right\} \right) } \,.
\]
The value of interest is $\perf_\lambda$, the supremum over all algorithms. We have only considered $\lambda = 1/2$.

It is easy to generalize 
\Cref{Result: Algorithm lowerbound} and \Cref{Result: Algorithms upperbound} for any $\lambda \in [0, 1]$. Indeed, we only need to generalize \Cref{Result: Choosing the maximum upperbound}, \Cref{Result: Characterization of the optimal algorithm} and \Cref{Result: Limit probability of algorithms}. All arguments within the proofs apply with no change.

\iftoggle{SODA}
{
    
}{
    For lower bounds, we generalize \Cref{Result: Algorithm lowerbound} by proving the generalization of \Cref{Result: Limit performance of RPI[x]} to the new goal for the DM. In other words, we prove the following result.
    \begin{lemma}[{Convexified limit performance of $\sRPI[x]$}]
        For all $x \in [0, 1]$, we have that
        \[
        \perf_\lambda(\sRPI[x]) 
            = - (1 - \lambda) x \ln(x) + \lambda x \left( \int_{x}^{1} \alpha \left( \frac{u}{1 + u} \right) \frac{1}{u^2} du + \alpha \left( \frac{1}{2} \right) \right) \,.
        \]
    \end{lemma}
    
    For upper bound, we first generalize \Cref{Result: Choosing the maximum upperbound} to the following.
    \begin{lemma}[Convexified choosing the maximum upperbound]
    \label{Result: Convexified choosing the maximum upperbound}
    	Consider the Repeated Prophet Inequality, with known or unknown values. Then, the limit $\lambda$-performance of any algorithm is bounded by a convex combination of accepting the maximum value in each phase. Formally, for any algorithm $\sALG$,
    	\begin{align*}
           	\perf_\lambda(\sALG) 
                \le \liminf_{n \to \infty} &(1 - \lambda) \PP(\ALG_n \text{ accepts the maximum value during Phase~1}) \\
                &+ \lambda \PP(\ALG_n \text{ accepts the maximum value during Phase~2}) \,.
        \end{align*}
    \end{lemma}
    
    Finally, the Waiting algorithms defined in \Cref{Section: Upper bound} for the Repeated Secretary Problem still contain the optimal algorithm, i.e. \Cref{Result: Characterization of the optimal algorithm} readily generalizes. Also, \Cref{Result: Limit probability of algorithms} extends and the formula of the limit convexified probability of choosing $v_1$ for the algorithm $\WAI[x]$ is given by the following result.
    \begin{lemma}[Convexified limit probability of algorithms]
    \label{Result: Convexified limit probability of algorithms}
        For all $x \in [0, 1]$, we have that
        \begin{align*}
            \lim_{n \to \infty} &(1 - \lambda) \PP( \WAI[x]_{n} \text{ accepts } v_1 \text{ during Phase~1}) 
                + \lambda \PP( \WAI[x]_{n} \text{ accepts } v_1 \text{ during Phase~2}) \\ 
                &= - (1 - \lambda) x \ln(x) + \lambda x \left( \int_{ \min(x, (e - 1)^{-1}) }^{(e - 1)^{-1}} \frac{ u + 1 }{ e u^2 } du 
                    + \int_{ \max(x, (e - 1)^{-1}) }^{1} \frac{1}{u} \ln \left( 1 + \frac{1}{u} \right) du  + \ln(2) \right) \,.
        \end{align*}
    \end{lemma}

    Therefore, we obtain the optimal algorithm for the $\lambda$-Repeated Secretary Problem by simply obtimizing over $x \in [0, 1]$ in the expression given by \Cref{Result: Convexified limit probability of algorithms}.
}

We show the numerical results obtained for various values of $\lambda \in [0, 1]$ in \Cref{Figure: Convexified bounds}. Note that the difference between the upper bound and our lower bound is mostly increasing on $\lambda$. This is expected since the value-oblivious variant of Repeated Prophet Inequality is a harder problem compared to the Repeated Secretary Problem, in the sense that the Repeated Secretary Problem is recovered if we constrain the adversary to choose a particular instance. This is formally stated in \Cref{Result: Choosing the maximum upperbound}. There are particular values of $\lambda$ that deserve some attention.

\begin{figure}[ht]
    \centering
    \includegraphics[width=0.5 \linewidth, height= 0.3 \linewidth]{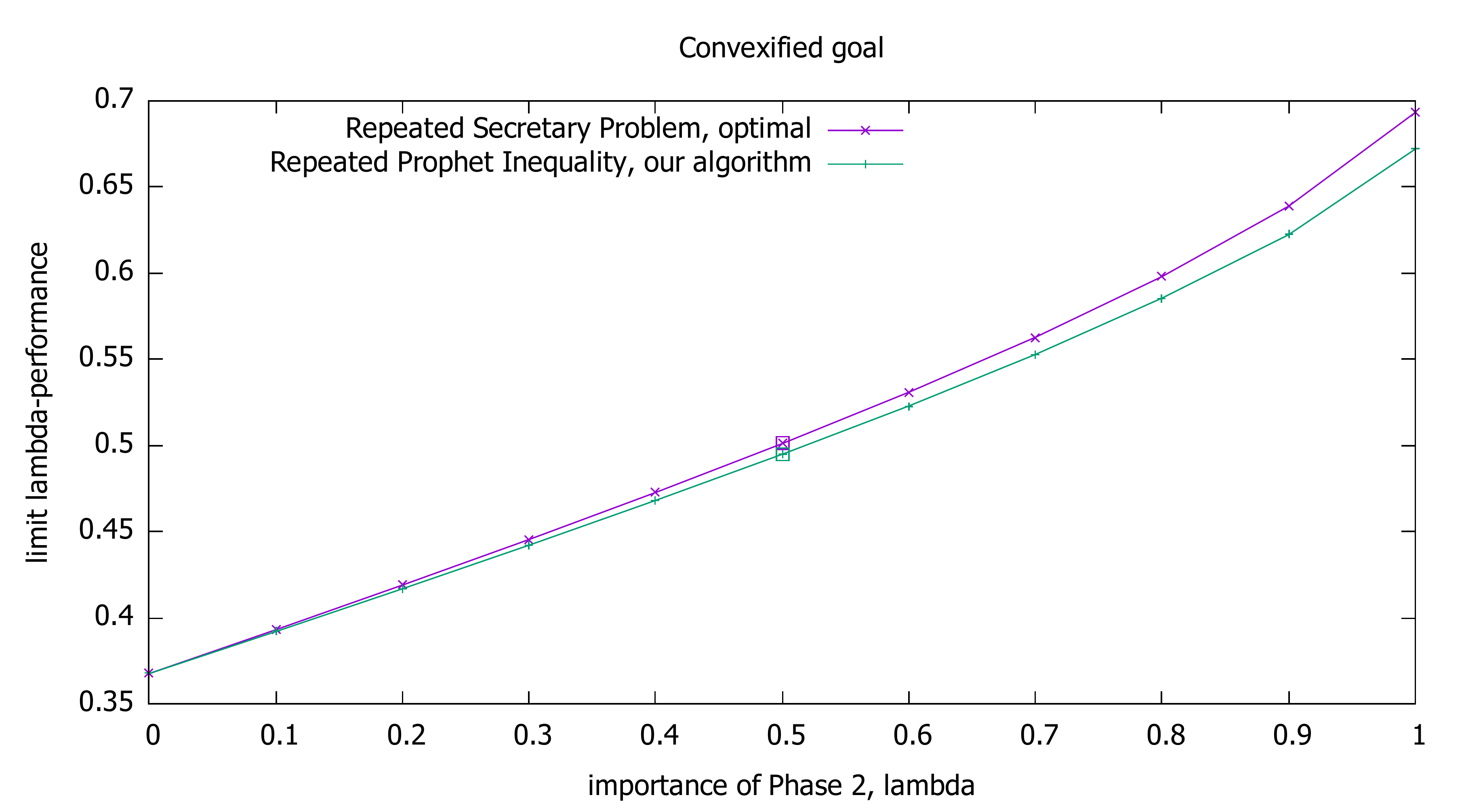}
    \caption{Lower and upper bounds for {$\perf_\lambda$}}
    \label{Figure: Convexified bounds}
\end{figure}

The case $\lambda = 0$ models the situation where the DM is indifferent about Phase~2, recovering the classical single-choice Prophet Inequality with no prior information. This problem is equivalent to the Secretary Problem~\cite{ferguson1989WhoSolvedSecretary} and the optimal performance is $1/e \approx 0.367$, due to~\cite{gilbert1966RecognizingMaximumSequence}. The case of $\lambda = 1$ models the situation where the DM is indifferent about Phase~1. In the case of the Repeated Prophet Inequality, we recover the Sample-driven Prophet Inequality where half of the items are sampled~\cite{correa2020SampledrivenOptimalStopping, kaplan2019CompetitiveAnalysisSample}. The optimal performance for this problem is approximately $0.671$, due to~\cite{correa2020SampledrivenOptimalStopping}. In the case of the Repeated Secretary Problem, we again recover the Secretary Problem where the first half of the items can not be choosen. By~\cite{gilbert1966RecognizingMaximumSequence}, the optimal performance for this model is $\ln(2) \approx 0.693$. In conclusion, our model recovers classical models and shows that we present a unifying approach.

\smallskip
\noindent{\bf Concluding remarks.} In this work we introduced the model of Repeated Prophet Inequality and present lower bound as an algorithm and upper bound results that show our algorithm is near-optimal. Interesting directions of future work include (a)~closing the optimality gap between the lower and upper bound; and (b)~generalization to Multi-phase Prophet Inequalities.

\iftoggle{SODA}
{
    \vspace{-1em}
    \section*{Acknowledgements} 
    \label{Section: Acknowledgements}
    \vspace{-0.5em}
}{
    \section*{Acknowledgements} 
    \label{Section: Acknowledgements}
}

This research was partially supported by the ERC CoG 863818 (ForM-SMArt) grant.

%

\bibliographystyle{alpha}
\bibliography{biblio}

%
%
%

\iftoggle{SODA}
{    
    \pagebreak
    \hspace{0pt}
    \vfill
    \begin{center}
        \Huge \textbf{Appendix}
    \end{center}
    \vfill
    \hspace{0pt}
    \thispagestyle{empty}
    \pagebreak

    \includepdf[pages=-]{main_full.pdf}
}

\end{document}